\documentclass[12pt]{amsart}

\usepackage{amsmath} 
\usepackage{amssymb,amscd,amsthm, verbatim,color,fancyhdr, mathrsfs} 
\usepackage{tikz}
\usepackage{subcaption}
\usepackage[a4paper, left=2.75cm, right=2.75cm, top=2.75cm,bottom=3.75cm,dvips]{geometry} 
\usepackage{placeins}
\usepackage{graphicx}
\usepackage{turnstile}
\usepackage{mathtools}
\usepackage{dhucs}
\usepackage[all]{xy}
\usepackage{float}
\usepackage{pgfplots}
\pgfplotsset{compat=newest}
\usetikzlibrary{calc, intersections}
\usepackage{times}
\usepackage[colorlinks=true,citecolor=blue]{hyperref}

\usepackage{enumitem} 

\usepackage{lineno,hyperref}
\modulolinenumbers[5]

\newtheorem*{claim*}{Claim}

\newtheorem{thm}{Theorem}[section]
\newtheorem{lemma}[thm]{Lemma}

\newtheorem{prop}[thm]{Proposition}
\newtheorem{obs}[thm]{Observation}

\theoremstyle{plain}
\theoremstyle{definition}
\newtheorem{eg}[thm]{Example}
\newtheorem{Def}[thm]{Definition}

\newcommand{\FF}{\mathcal{F}}
\newcommand{\LL}{\mathcal{L}}
\newcommand{\PP}{\mathcal{P}}
\newcommand{\HH}{\mathcal{H}}

\newcommand{\R}{\mathbb{R}}

\newcommand{\calP}{\mathcal{P}}

\newcommand{\calF}{\mathcal{F}}

\newcommand{\gj}{\mathcal{G}}

\newcommand{\arc}[1]{\gamma_S(#1)}

\title{On a conjecture of Karasev}

\author{Seunghun~Lee}
\address{S.~Lee \\ Department of Mathematical Sciences \\ KAIST \\ Daejeon \\ South Korea} \email{prosolver@kaist.ac.kr}

\author{Kangmin~Yoo}
\address{K.~Yoo \\ Department of Mathematical Sciences \\ KAIST \\ Daejeon \\ South Korea} \email{dearmymind@kaist.ac.kr}

\thanks{This research was supported by Basic Science Research Program through the National Research Foundation of Korea(NRF) funded by the Ministry of Education (NRF-2016R1D1A1B03930998).}

\begin{document}
	\begin{abstract}
		Karasev conjectured that for any set of $3k$ lines in general position in the plane, which is partitioned into $3$ color classes of equal size $k$, the set can be partitioned into $k$ colorful 3-subsets such that all the triangles formed by the subsets have a point in common. Although the general conjecture is false, we show that Karasev's conjecture is true for lines in convex position. We also discuss possible generalizations of this result.
	\end{abstract}
	
	\maketitle

	\section{Introduction}\label{section-intro}
	One of the classical results in discrete geometry is Tverberg's Theorem \cite{history0} which asserts that any set of $(d+1)(r-1)+1$ points in $\R^d$ can be partitioned into $r$ disjoint subsets whose convex hulls have a point in common. Tverberg's theorem has many variants as we discuss in this section.
	
	\subsection{Dual Tverberg theorems}
	We say that a collection of hyperplanes in $\R^d$ is \textit{in general position} if the intersection of any $d$ of the hyperplanes is a single point, but the intersection of any $d+1$ of them is empty. Note that, for any $d+1$ hyperplanes in general position in $\R^d$, there is a unique bounded $d$-dimensional simplex formed by the hyperplanes. There are dual versions of Tverberg's theorem, which consider a set of hyperplanes in general position and bounded simplices formed by its subsets of size $d+1$. For a set of lines, Roudneff \cite{roudneff} proved the following result.
	
	\begin{thm}
		For a set of $3r$ lines in general position in the plane, the set can be partitioned into $r$ subsets of size 3 so that all the triangles formed by the subsets have a point in common.
	\end{thm}
	
	And Karasev \cite{dualcentral} proved a generalization of this result.
	
	\begin{thm}
		\label{dual-tve}
		Let $r$ be a prime power. For a set of $(d+1)r$ hyperplanes in general position in $\R^d$, the set can be partitioned into $r$ subsets of size $d+1$ so that all the $d$-dimensional simplices formed by the subsets have a point in common.
	\end{thm}
	
	\subsection{Colored Tverberg theorems}
	For a set $P$ which is partitioned into color classes, a subset $Q$ of $P$ is called \textit{colorful} if it contains exactly one from each color class. A research direction which has recieved a significant amount of attention in recent years is to establish a colored version of Tverberg's theorem.
	
	\begin{thm}[The colored Tverberg theorem]
		\label{col-tve}
		For any $d$ and $r$ with $d \ge 1$ and $r \ge 2$, there exists a positive integer $t$ which satisfies the following: For a set of $(d+1)t$ points in $\R^d$ which is partitioned into $d+1$ color classes of size $t$, there exist $r$ disjoint colorful subsets whose convex hulls have a point in common.
	\end{thm}
	
	This theorem was originally a conjecture of B\'{a}r\'{a}ny, F\"{u}redi and Lov\'{a}sz \cite{history1}. Let $t(d,r)$ be the smallest integer for which the conclusion of Theorem \ref{col-tve} holds. A particular case $t(2,3) \le 7$ was given in the same paper by B\'{a}r\'{a}ny, F\"{u}redi and Lov\'{a}sz \cite{history1}. Later, B\'{a}r\'{a}ny and Larman \cite{history2} proved $t(2,r) = r$ and conjectured $t(d,r) = r$ for every dimension $d$. Using topological methods, the general case was first proved by \v{Z}ivaljevi\'{c} and Vre\'{c}ica \cite{history3}, who also showed that $t(d,r) \le 2r-1$ whenever $r$ is a prime number. The same bound was extended to all $r$ which are prime powers \cite{user}. Later Blagojevi\'{c}, Matschke and Ziegler \cite{history4} obtained the optimal bound $t(d,r) = r$ whenever $r+1$ is a prime number.
	
	\subsection{The dual colored Tverberg theorem and our main result}
	Karasev \cite{dualcentral} also proved a dual version of the colored Tverberg theorem.
	
	\begin{thm}[The dual colored Tverberg theorem]
		\label{dual-col-tve}
		For any $d$ and $r$ with $d \ge 1$ and $r \ge 2$, there exists a positive integer $k$ which satisfies the following: For a set of $(d+1)k$ hyperplanes in $\R^d$ in general position, which is partitioned into $d+1$ color classes of size $k$, there exist $r$ disjoint colorful subsets  such that all the $d$-dimensional simplices formed by the subsets have a point in common.
	\end{thm}
	
	Similarly with $t(d,r)$, let $k(d,r)$ be the smallest integer for which the conclusion of Theorem \ref{dual-col-tve} holds.
	In \cite{dualcentral}, Karasev proved $k(d,r) \le 2r-1$ when $r$ is a prime power. Especially for the $2$-dimensional case, he conjectured that the equality $k(2,r) = r$ holds, which is an analogue of $t(2,r) = r$ for the colored Tverberg theorem. However, the following example shows that $k(2,r)\ne r$ whenever $r=2t$ for some odd number $t$ (The counterexample showing $k(2,2)\ne 2$ was also discovered independently by Liping~Yuan \cite{yuan}).
	
	\begin{eg}
		Consider the set $\LL$ with 6 lines shown in Figure \ref{fig-ceg}. Note that the set $\LL$ is in general position and each line in $\LL$ is colored by either RED, BLUE, or GREEN. There are 4 ways to partition $\LL$ into 2 colorful subsets, and in every case the subsets form two disjoint triangles. Therefore, the set $\LL$ is a counterexample to the conjecture.
		
		\begin{figure}[ht]
			\begin{center}
				\begin{subfigure}{0.45\textwidth}
					\begin{tikzpicture}[scale=0.7]
					\draw[fill,gray!30] (-1.22,0.03)--(-0.52,1.4)--(2.8,0.03)--(-1.22,0.03);
					\draw[fill,gray!30] (-0.03,-0.52)--(-0.03,-1.73)--(-0.42,-0.73)--(0.03,-0.52);
					
					\draw [red,densely dotted,thick] (-3,0) -- (4,0);
					\draw [red,densely dotted,thick] (0,3.5) -- (0,-3);
					\draw [blue,dashed,thick] (-3,2.5) -- (4,-0.5);
					\draw [blue,dashed,thick] (-2,3) -- (0.5,-3);
					\draw [green] (-2.5,-2.5) -- (0.5,3.5);
					\draw [green] (-3,-2) -- (3,1);
					
					\node [above] at (-3,0) {$l_1$};
					\node [above] at (0,3.5) {$l_4$};
					\node [above] at (-2,3) {$l_5$};
					\node [above] at (-3,2.5) {$l_2$};
					\node [right] at (0.5,3.5) {$l_3$};
					\node [above] at (3,1) {$l_6$};
					\end{tikzpicture}
					\subcaption{$ \{l_1,l_2,l_3\} \cup \{l_4,l_5,l_6\} $}
				\end{subfigure}
				~
				\begin{subfigure}{0.45\textwidth}
					\begin{tikzpicture}[scale=0.7]
					\draw[fill,gray!30] (1.08,0.03)--(2.75,0.03)--(1.82,0.41)--(1.08,0.03);
					\draw[fill,gray!30] (-0.03,2.4)--(-0.03,-1.72)--(-0.95,0.53)--(-0.03,2.39);
					
					\draw [red,densely dotted,thick] (-3,0) -- (4,0);
					\draw [red,densely dotted,thick] (0,3.5) -- (0,-3);
					\draw [blue,dashed,thick] (-3,2.5) -- (4,-0.5);
					\draw [blue,dashed,thick] (-2,3) -- (0.5,-3);
					\draw [green] (-2.5,-2.5) -- (0.5,3.5);
					\draw [green] (-3,-2) -- (3,1);
					
					\node [above] at (-3,0) {$l_1$};
					\node [above] at (0,3.5) {$l_4$};
					\node [above] at (-2,3) {$l_5$};
					\node [above] at (-3,2.5) {$l_2$};
					\node [right] at (0.5,3.5) {$l_3$};
					\node [above] at (3,1) {$l_6$};
					\end{tikzpicture}
					\subcaption{$ \{l_1,l_2,l_6\} \cup \{l_3,l_4,l_5\} $}
				\end{subfigure}
				\vskip\baselineskip
				\begin{subfigure}{0.45\textwidth}
					\begin{tikzpicture}[scale=0.7]
					\draw[fill,gray!30] (-1.22,0.03)--(-0.98,0.5)--(-0.8,0.03)--(-1.22,0.03);
					\draw[fill,gray!30] (0.03,-0.5)--(0.03,1.18)--(1.77,0.42)--(0.03,-0.45);
					
					\draw [red,densely dotted,thick] (-3,0) -- (4,0);
					\draw [red,densely dotted,thick] (0,3.5) -- (0,-3);
					\draw [blue,dashed,thick] (-3,2.5) -- (4,-0.5);
					\draw [blue,dashed,thick] (-2,3) -- (0.5,-3);
					\draw [green] (-2.5,-2.5) -- (0.5,3.5);
					\draw [green] (-3,-2) -- (3,1);
					
					\node [above] at (-3,0) {$l_1$};
					\node [above] at (0,3.5) {$l_4$};
					\node [above] at (-2,3) {$l_5$};
					\node [above] at (-3,2.5) {$l_2$};
					\node [right] at (0.5,3.5) {$l_3$};
					\node [above] at (3,1) {$l_6$};
					\end{tikzpicture}
					\subcaption{$ \{l_1,l_3,l_5\} \cup \{l_2,l_4,l_6\} $}
				\end{subfigure}
				~
				\begin{subfigure}{0.45\textwidth}
					\begin{tikzpicture}[scale=0.7]
					\draw[fill,gray!30] (-0.72,-0.03)--(-0.43,-0.69)--(0.89,-0.03)--(-0.72,-0.03);
					\draw[fill,gray!30] (-0.03,1.22)--(-0.03,2.4)--(-0.49,1.46)--(0.03,1.22);
					
					\draw [red,densely dotted,thick] (-3,0) -- (4,0);
					\draw [red,densely dotted,thick] (0,3.5) -- (0,-3);
					\draw [blue,dashed,thick] (-3,2.5) -- (4,-0.5);
					\draw [blue,dashed,thick] (-2,3) -- (0.5,-3);
					\draw [green] (-2.5,-2.5) -- (0.5,3.5);
					\draw [green] (-3,-2) -- (3,1);
					
					\node [above] at (-3,0) {$l_1$};
					\node [above] at (0,3.5) {$l_4$};
					\node [above] at (-2,3) {$l_5$};
					\node [above] at (-3,2.5) {$l_2$};
					\node [right] at (0.5,3.5) {$l_3$};
					\node [above] at (3,1) {$l_6$};
					\end{tikzpicture}
					\subcaption{$ \{l_1,l_5,l_6\} \cup \{l_2,l_3,l_4\} $}
				\end{subfigure}
			\end{center}
			\caption{A counterexample to the conjecture $k(2,r) = r$.}
			\label{fig-ceg}
		\end{figure}
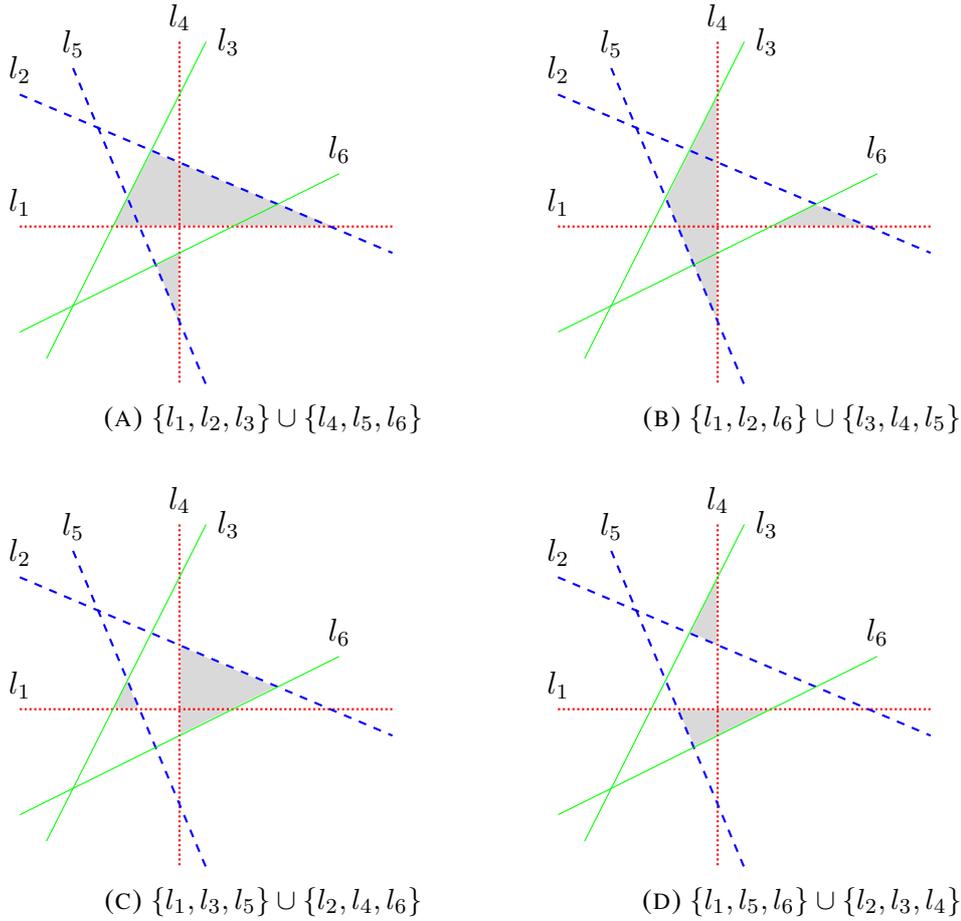
		
		Now we make an arbitrarily large counterexample. For a fixed odd number $t$, we make $t$ copies for each line in $\LL$. The copied lines are perturbed a little from the original lines so that the copies form a new set $\LL^{(t)}$ of lines in general position, and every selection of copies, one for each original line, gives a set of lines whose arrangement is isomorphic to that of $\LL$. Also, each copy is painted by the same color with its original line.
		
		To show that $\LL^{(t)}$ is a counterexample, it is sufficient to show that in any partition of $\LL^{(t)}$ into colorful subsets, there are two parts $L_1$ and $L_2$ such that every line $l$ in $\LL$ has a copy in $L_1 \cup L_2$, which makes $L_1\cup L_2$ isomorphic to $\LL$ as a line arrangement. Combinatorially, this is same as asking whether it is possible to choose $2t$ facets from the regular octahedron, possibly multiple times, so that there are no pairs of opposite facets among selected ones and each vertex is covered by the facets exactly $t$ times. The only possible way to do this is first taking 4 facets so that any two of them interset exactly at a vertex, and choosing each of them exactly $t/2$ times. And this is impossible when $t$ is odd.
	\end{eg}

	Even though the general conjecture is false, we can show that Karasev's conjecture is true when we add an additional condition on the arrangement of lines. We say that lines $l_1, \dots, l_n$ in the plane $\R^2$ are \textit{in convex position} if they are in general position and the complement $\R^2 \setminus (\bigcup_{i=1}^n l_i)$ has a connected component whose boundary meets every line. And here is our main result.
	
	\begin{thm}\label{thm-line}
		For any set of $3k$ lines in $\R^2$ in convex position which is partitioned into $3$ color classes of size $k$, the set can be partitioned into $k$ colorful subsets such that all the triangles formed by the subsets have a point in common.
	\end{thm}
	
	
	This note is organized as follows. In Section \ref{section-pre}, we reduce Theorem \ref{thm-line} to its dual version Theorem \ref{thm1} below and make some observations for the proof of Theorem \ref{thm1}. In Section \ref{section-proof}, we prove Theorem \ref{thm1}. Finally in Section \ref{section-final}, we discuss possible generalizations of Theorem \ref{thm-line}.
	
	\section{Preliminaries}\label{section-pre}
	In this section, we reduce Theorem \ref{thm-line} to its dual version Theorem \ref{thm1} below and make some observations which are useful in the proof of Theorem \ref{thm1}. Before that, let us introduce some conventions we use in the remainder of this note.
	
	By \textit{a $k$-set} (or \textit{a $k$-subset}), we mean a set (or a subset, respectively) of size $k$. For a set $P$ which is partitioned into color classes, a subset $Q$ of $P$ is called \textit{colorful} if it contains exactly one from each color class. And when we consider a partition of a set, we naturally identify it with the family of all parts from the partition.
	
	\subsection{Reduction to a dual version on a circle}
	In this subsection, we state a dual version of Theorem \ref{thm-line} regarding points on a circle, and show how it implies Theorem \ref{thm-line}. 
	\medskip

	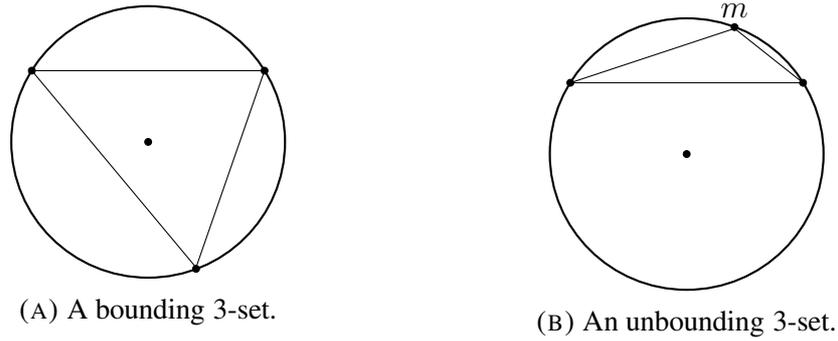
\begin{figure}[ht]
		\begin{center}
			\begin{subfigure}{0.45\textwidth}
				\centering
				\begin{tikzpicture}[scale=0.9]
				\draw [thick] (0,0) circle [radius = 2];
				
				\draw[fill] (0,0) circle [radius = 0.05];
				
				\draw (-1.7,1.05) -- (1.7,1.05);
				\draw (1.7,1.05) -- (0.7,-1.85);
				\draw (0.7,-1.85) -- (-1.7,1.05);
				
				\draw [fill] (-1.7,1.05) circle [radius = 0.05];
				\draw[fill] (1.7,1.05) circle [radius = 0.05];
				\draw[fill] (0.7,-1.87) circle [radius = 0.05];
				
				\end{tikzpicture}
				\subcaption{A bounding 3-set.}
			\end{subfigure}
			~
			\begin{subfigure}{0.45\textwidth}
				\centering
				\begin{tikzpicture}[scale=0.9]
				\draw [thick] (0,0) circle [radius = 2];
				
				\draw[fill] (0,0) circle [radius = 0.05];
				
				\draw (-1.7,1.05) -- (1.7,1.05);
				\draw (1.7,1.05) -- (0.7,1.85);
				\draw (0.7,1.85) -- (-1.7,1.05);
				
				\draw [fill] (-1.7,1.05) circle [radius = 0.05];
				\draw[fill] (1.7,1.05) circle [radius = 0.05];
				\draw[fill] (0.7,1.87) circle [radius = 0.05];
				\node[above] at (0.7,1.85) {$m$};
				
				\end{tikzpicture}
				\subcaption{An unbounding 3-set.}
			\end{subfigure}
		\end{center}
		\caption{Two types of 3-sets on a circle.}
		\label{fig-def-triple}
	\end{figure}
	
	A 3-set of points on a circle without antipodal pairs is said to be \textit{bounding} if its convex hull contains the center of the circle, and \textit{unbounding} otherwise.
	For any unbounding 3-sets, there is a unique point in the set which is contained in the shorter arc connecting the other two points, and we call this a \textit{middle point} of the 3-set (In Figure \ref{fig-def-triple}-(B), the middle point is labelled by $m$).
	
	Let $P$ be a set of $3k$ points on a circle $S$ partitioned into 3 color classes of equal size, and let $\calP$ be a partition of $P$ into disjoint 3-subsets. The partition $\calP$ is said to be \textit{colorful} if all parts of the partition $\calP$ are colorful. If a point $p$ is a middle point of some unbounding 3-set in $\calP$, then we say that the point $p$ is a \textit{middle point of the partition $\calP$}. Finally, we say that the middle points of a partition $\calP$ are \textit{consecutive} if
	\begin{itemize}
		\item there is a semicircle in $S$ containing all middle points of $\calP$, and
		\item the shortest closed arc in $S$ containing all middle points has no other points from $P$.
	\end{itemize}
	For convenience, we also say that the middle points of $\calP$ are consecutive even in the case when there are no middle points of $\PP$.
	
	\medskip
	
	Now, we state the dual version of Theorem \ref{thm-line}.
	
	\begin{thm}\label{thm1}
		Let $P$ be a set of $3k$ points on a circle without antipodal pairs, which is partitioned into $3$ color classes of size $k$. Then there exists a partition of $P$ into $k$ disjoint colorful 3-sets whose middle points are consecutive.
	\end{thm}
	
	Before proving this theorem in the next section, we first show that Theorem \ref{thm1} implies Theorem \ref{thm-line}.
	
	\begin{proof}[Proof of Theorem \ref{thm-line}]
		Let $\LL$ be a given set of $3k$ lines in convex position.
		We choose a sufficiently small circle $S$ so that it is contained in a connected component of the complement $\R^2 \setminus (\bigcup_{l \in \LL} l)$ which meets every line in $\LL$. For each line $l$ in $\LL$, let $p(l)$ be the intersection point of $S$ with a ray from the center of $S$, which perpendicularly intersects $l$. And paint the point $p(l)$ with the color of $l$.
		If we let $P =\{p(l) : l \in \LL\}$, then $P$ has no antipodal pairs and is partitioned into $3$ color classes.
		
		Let $\calP$ be a partition of $P$ given by Theorem \ref{thm1}. With respect to the partition $\calP$, we make each line $l$ in $\FF$ correspond to a halfspace $H(l)$ in the following way: 
		\begin{itemize}
			\item If $p(l)$ is a middle point of $\calP$, then let $H(l)$ be the closed halfplane bounded by $l$ which does not contain the center of $S$.
			\item Otherwise, let $H(l)$ be the closed halfplane bounded by $l$ which contains the center.
		\end{itemize} 
		Then for each part $\{p(l_1),p(l_2),p(l_3)\}$ in the partition $\calP$, the 3-set $\{l_1, l_2, l_3\}$ is colorful and the triangle formed by $\{l_1, l_2, l_3\}$ is equal to the intersection of halfplanes $H(l_1) \cap H(l_2) \cap H(l_3)$.
		
		So, it is sufficient to show that the intersection $\bigcap_{l \in \LL}H(l)$ is nonempty. If there is at most one middle point of $\PP$, then this is obvious. So we assume that there are at least two middle points of $\PP$. Let $\gamma$ be the shortest closed arc in $S$ containing all the middle points of $\calP$, and let $l_a$ and $l_b$ be the members in $\LL$ such that $p(l_a)$ and $p(l_b)$ are the boundary points of $\gamma$. And let $q$ be the intersection point of $l_a$ and $l_b$. Since the arc $\gamma$ is contained in a semicircle, for each $l \in \LL$, the halfplane $H(l)$ contains the common point $q$.
	\end{proof}

	\subsection{A colorful partition with respect to a circular ordering}
	In this subsection, we consider some observations which guarantee the existence of a colorful partition with respect to a circular ordering. We first recall a concept, which was introduced and investigated by B\'{a}r\'{a}ny, Holmsen and Karasev \cite{join1}.
	\begin{Def}
		The \textbf{geometric join} of $m$ point sets $C_1, \dots, C_m$ in the Euclidean space $\R^d$, denoted by $\gj(C_1, \dots, C_m)$, 
		is the set of all convex combinations $\sum_{j=1}^m t_jp_j \in \R^d$ where $p_j \in C_j$, $t_j \geq 0$ and $\sum_{j=1}^m t_j =1$.
	\end{Def}
	
	In particular, in the plane, it was shown that the geometric join of 3 point sets is starshaped \cite{join2, join3}. Using this, we can prove following proposition.
	
	\begin{prop}
		\label{join}
		Let $P$ be a set of points on a unit circle $S$ centered at the origin $O$, partitioned into 3 nonempty color classes $C_1$, $C_2$ and $C_3$. If all colorful 3-subsets of $P$ are unbounding, then there is a point $q$ on $S$ such that the closed line segment connecting $q$ and $O$ is disjoint from $\gj(C_1,C_2, C_3)$.
	\end{prop}
	\begin{proof}
		Since $\gj(C_1,C_2, C_3)$ is starshaped, we can find a point $p$ in $\gj(C_1,C_2, C_3)$ so that the point $p$ can see every point in $\gj(C_1,C_2, C_3)$. When $t$ is a nonnegative scalar, $t\cdot (-p)$ is not contained in $\gj(C_1,C_2, C_3)$. Otherwise the origin $O$ must be contained in $\gj(C_1,C_2, C_3)$, which implies there is a colorful bounding 3-set in $P$ leading to a contradiction. Therefore, we can choose the point $-p/||p||$ as $q$ on $S$.
	\end{proof}
	
	The following proposition gives a combinatorial characterization to have a certain colorful partition.
	
	\begin{prop}\label{matching}
		Let $P$ be a set of $3k$ elements, which is partitioned into 3 subsets $C_1$, $C_2$ and $C_3$ of size $k$. Suppose that there is another partition of $P$ into 3 subsets $A_1$, $A_2$ and $A_3$ of size $k$. Then, we can find a partition of $P$ into $k$ 3-sets $T_1, \dots, T_k$ such that $|T_i \cap A_j|=|T_i \cap C_j|=1$ for all $i\in \{1,\dots, k\}$ and $j \in \{1,2,3\}$.
	\end{prop}
	\begin{proof}
		We use induction on $k$. Since it is obvious when $k=1$, suppose that $k>1$. Define a bipartite graph $G$ with $V(G) = \{A_1,A_2,A_3\}\,\cup\,\{C_1,C_2,C_3\}$  and $E(G) = \{A_iC_j : A_i \cap C_j \ne \emptyset\}$.
		By the Pigeonhole principle, union of any $t$ of $C_1$, $C_2$, $C_3$ intersects with at least $t$ of $A_1$, $A_2$, $A_3$ for each $t\in \{1,2,3\}$.
		Thus there exists a perfect matching $M$ in $G$ by Hall's theorem \cite{hall}. Choose $x_i \in A_i \cap C_j$ for each $A_iC_j \in M$, then let $T = \{x_1, x_2, x_3\}$.
		Then we have $|T\cap A_j|=|T\cap C_j|=1$ for every $j \in \{1,2,3\}$. By adding $T$ to the partition for $P\setminus T$ obtained by the induction hypothesis, we can get the desired partition.
	\end{proof}

	\begin{figure}[ht]
		\begin{center}
			\begin{subfigure}[{caption for a}]{0.45\textwidth}
				\centering
				\begin{tikzpicture}[scale=0.9]
				\draw[thick] (0,0) circle [radius = 2];
				\draw[fill] (0,0) circle [radius = 0.05];
				\draw[ thick,->] (0,-2) arc [radius = 2, start angle = -90, end angle = -110];
				
				\draw[dotted] (-1.9,-0.7)--(0.8,1.85);
				\draw[dotted] (-0.5,1.95)--(1.75,-1);
				
				\draw[ultra thick] (1.75,-1) arc [radius = 2.02, start angle = -30, end angle = 200];
				
				\node[right] at (1.75,-1) {$p_6$};
				\draw[green,fill,shift={(1.75,-1)}] (0,0) circle [radius = 0.15];
				
				\node[left] at (-1.9,-0.7) {$p_1$};
				\draw[green,fill,shift={(-1.9,-0.7)}] (0,0) circle [radius = 0.15];
				
				\node[right] at (1.8,0.95) {$p_5$};
				\draw [red,fill,shift={(1.8,0.95)}] (0,0.13) -- (0.13,-0.1) -- (-0.13,-0.1) -- (0,0.13);
				
				\node[above] at (-0.5,1.95) {$p_3$};
				\draw [red,fill,shift={(-0.5,1.95)}] (0,0.13) -- (0.13,-0.1) -- (-0.13,-0.1) -- (0,0.13);
				
				\node[above] at (0.8,1.85) {$p_4$};
				\draw [blue,fill,shift={(0.8,1.85)}] (-0.11,-0.11) rectangle (0.11,0.11);
				
				\node[left] at (-1.6,1.25) {$p_2$};
				\draw [blue,fill,shift={(-1.6,1.25)}] (-0.11,-0.11) rectangle (0.11,0.11);
				
				\node[below] at (0,-1.5) {$p$};
				\draw[fill] (0,-2) circle [radius = 0.06];
				\end{tikzpicture}
				\subcaption{Points satisfying the both conditions.}
				\label{fig-lem2-1}
			\end{subfigure}
			~\begin{subfigure}[{caption for a}]{0.45\textwidth}
				\centering
				\begin{tikzpicture}[scale=0.9]
				\draw[thick] (0,0) circle [radius = 2];
				\draw[fill] (0,0) circle [radius = 0.05];				
				
				\draw (1.75,-1)--(-0.5,1.95)--(-1.6,1.25)--(1.75,-1);
				
				\draw (1.8,0.95)--(0.8,1.85)--(-1.9,-0.7)--(1.8,0.95);
				
				\node[right] at (1.75,-1) {$p_6$};
				\draw[green,fill,shift={(1.75,-1)}] (0,0) circle [radius = 0.15];
				
				\node[left] at (-1.9,-0.7) {$p_1$};
				\draw[green,fill,shift={(-1.9,-0.7)}] (0,0) circle [radius = 0.15];
				
				\node[right] at (1.8,0.95) {$p_5$};
				\draw [red,fill,shift={(1.8,0.95)}] (0,0.13) -- (0.13,-0.1) -- (-0.13,-0.1) -- (0,0.13);
				
				\node[above] at (-0.5,1.95) {$p_3$};
				\draw [red,fill,shift={(-0.5,1.95)}] (0,0.13) -- (0.13,-0.1) -- (-0.13,-0.1) -- (0,0.13);
				
				\node[above] at (0.8,1.85) {$p_4$};
				\draw [blue,fill,shift={(0.8,1.85)}] (-0.11,-0.11) rectangle (0.11,0.11);
				
				\node[left] at (-1.6,1.25) {$p_2$};
				\draw [blue,fill,shift={(-1.6,1.25)}] (-0.11,-0.11) rectangle (0.11,0.11);
				
				\end{tikzpicture}
				\subcaption{The obtained colorful partition.}
			\end{subfigure}
		\end{center}
		\caption{An example for Observation \ref{lem2} and \ref{obs2} when $k=2$. The geometric join of the color classes is described by its radial projection to the circle as a thick arc.}
		\label{fig-lem2}
	\end{figure}
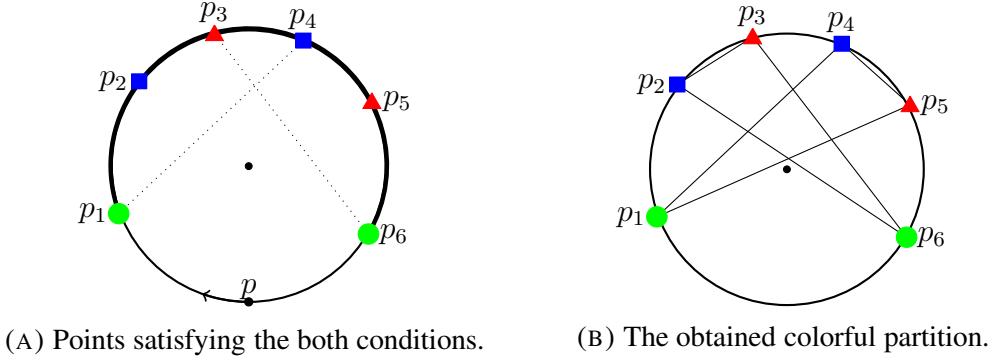
	
	Now, let $P$ be a set of $3k$ points on a unit circle $S$ centered at the origin $O$, partitioned into 3 color classes $C_1$, $C_2$ and $C_3$ of size $k$. We say that a set of $n$ points on a circle has \textit{a clockwise ordering $p_1,\dots, p_n$ from the point $q$} if, as we go along the circle $S$ clockwise from the point $q$, $p_i$ is the $i$th point we meet during the tour. For two non-antipodal points $p$ and $q$ on a circle $S$, the \textit{arc moving from $p$ to $q$ clockwise} is denoted by $\arc{p,q}$.

	\begin{obs}\label{lem2}
		If there are no colorful bounding 3-subsets of $P$, then we can find a colorful partition of $P$ with consecutive middle points.
	\end{obs}	
	\begin{proof}
		Since there are only unbounding 3-sets among colorful subsets of $P$, we can apply Proposition \ref{join} to $P$ in order to get the point $q$ in the conclusion of Proposition \ref{join}. Let $p_1, \dots, p_{3k}$ be the clockwise ordering from $q$. And let $A_j=\{p_i \,:\, (j-1)k+1\leq i \leq jk\}$ for $j\in \{1,2,3\}$. By using Proposition \ref{matching} on two partitions $\{A_1, A_2, A_3\}$ and $\{C_1, C_2, C_3\}$, we can find a colorful partition $\PP$ into 3-subsets such that each part uses exactly one from each $A_j$. Note that each part must be unbounding by the assumption. And the assumption that the line segment connecting $O$ and $p$ does not have intersection with $\gj(C_1, C_2, C_3)$ concludes that only points in $A_2$ can be used as a middle point of the partition. This readily implies that the middle points of $\PP$, which are the points in $A_2$, are consecutive.
	\end{proof}
	
	We need one last observation. As in the proof of Observation \ref{lem2}, we assume that points in $P$ has a clockwise ordering $p_1, p_2, \dots, p_{3k}$ from some point, and let $A_j=\{p_i \,:\, (j-1)k+1\leq i \leq jk\}$ for $j\in \{1,2,3\}$.
	\begin{obs}
		\label{obs2}
		Suppose that each of the arcs $\arc{p_1, p_{2k}}$ and $\arc{p_{k+1},p_{3k}}$ are strictly contained in a semicircle (but not necessarily the same semicircle). Then there exists a colorful partition $\calP$ of the set $P$ such that each part of $\PP$ contains exactly one from each $A_j$, and the set of middle points of $\PP$ is contained in $A_2$.
	\end{obs}
	\begin{proof}
		By using Proposition \ref{matching} on two partitions $\{A_1, A_2, A_3\}$ and $\{C_1, C_2, C_3\}$, we can find a colorful partition into 3-subsets such that each part uses exactly one from each $A_j$.
		
		Suppose that one of the parts $T$ is unbounding. Let us denote the unique point in $T\cap A_j$ by $x_j$. The point $x_3$ cannot be a middle point of $T$, since $\arc{x_1, x_2}$ does not contain $x_3$ but it is contained in $\arc{p_1, p_{2k}}$ which implies that $\arc{x_1, x_2}$ is the shorter arc connecting $x_1$ and $x_2$. Similarly $x_1$ cannot be a middle point of $T$, so $x_2$ is a middle point of $T$.
	\end{proof}
	
	\section{Proof of Theorem \ref{thm1}}\label{section-proof}
	In this section, we prove Theorem \ref{thm1}.
	Throughout this section, let $S$ be the unit circle centered at the origin $O$, and $P$ be a set of $3k$ points on $S$ without antipodal pairs of points.
	Assume that $P$ is partitioned into 3 color classes $C_1$, $C_2$ and $C_3$ of size $k$.
	
	\medskip
	
	Recall that we need to find a colorful partition of $P$ whose middle points are consecutive.
	If there exists a colorful partition of $P$ which contains at most one colorful unbounding $3$-set, then we are done.
	So we assume that any colorful partition of $P$ contains at least two colorful unbounding $3$-sets.
	
	If the middle points of a colorful partition $\calP$ of $P$ are contained in a semicircle, then the shortest closed arc containing all the middle points of $\calP$ is well-defined.
	We denote this arc by $\gamma(\calP)$.
	In particular, the boundary points of $\gamma(\calP)$ are middle points of two distinct unbounding $3$-sets in $\calP$.
	Then the rest of proof of Theorem \ref{thm1} can be drawn by the following lemma.
	
	\begin{lemma}\label{lem3}
		Let $\calP$ be a colorful partition of $P$ which satisfies the following conditions:
		\begin{enumerate}[label=(\roman*)]
			\item\label{lem3-cond1}\label{minimalitycalP}
			$\calP$ contains the minimum number of colorful unbounding $3$-sets,
			
			\item\label{lem3-cond2}\label{semicircled}
			The middle points of $\calP$ are contained in a semicircle.
			
			\item\label{lem3-cond3}\label{consecutiveness}
			If a point in $P$ is contained in $\gamma(\calP)$, then it is either a middle point of $\calP$ or a point from a colorful bounding $3$-set in $\calP$.
		\end{enumerate}
		If $\gamma(\calP)$ contains a point from a colorful bounding $3$-set in $\calP$, then there exists another colorful partition $\calP'$ of $P$ satisfying \ref{lem3-cond1}, \ref{lem3-cond2} and \ref{lem3-cond3}, such that $\gamma(\calP')$ is strictly contained in $\gamma(\calP)$.
	\end{lemma}
	
	\begin{proof}[Proof of Theorem \ref{thm1}]
		First, we show that there exists at least one colorful partition of $P$ satisfying \ref{lem3-cond1}, \ref{lem3-cond2} and \ref{lem3-cond3}.
		Let $\calP$ be a colorful partition of $P$ which satisfies \ref{lem3-cond1}.
		Let $\calP_U$ be the subfamily of $\calP$ which consists of all unbounding $3$-sets in $\calP$, and let $X$ be the union of all $3$-sets in $\calP_U$.
		In particular, $\calP_U \subseteq \calP$ and $X \subseteq P$. Note that there is no colorful bounding $3$-subset of $X$ by \ref{lem3-cond1}. So we can use Observation \ref{lem2} to obtain a colorful partition $\calP'_U$ of $X$ with consecutive middle points.
		In particular, the middle points of $\calP'_U$ are contained in a semicircle.
		If we replace $\calP_U$ with $\calP'_U$ in $\PP$, then we obtain a new colorful partition $\calP'$ of the whole set $P$, which satisfies \ref{lem3-cond1}, \ref{lem3-cond2} and \ref{lem3-cond3}.
		
		Now let $\calF$ be the collection of all colorful partitions of $P$ satisfying \ref{lem3-cond1}, \ref{lem3-cond2} and \ref{lem3-cond3}.
		Choose $\calP_0 \in \calF$ so that $|\gamma(\calP_0) \cap P|$ is minimum possible among all partitions in $\calF$, i.e. $|\gamma(\calP_0) \cap P|$ $=$ $\min\{|\gamma(\calP) \cap P| : \calP \in \calF\}$.
		We claim that the middle points of $\calP_0$ are consecutive. Suppose otherwise.
		Then there is a point $p \in P$ contained in $\gamma(\calP_0)$ which is not a middle point of $\calP_0$.
		Thus $p$ is a point from a colorful bounding $3$-set in $\calP_0$ by \ref{lem3-cond3}.
		By Lemma \ref{lem3}, there exists another partition $\calP_1 \in \calF$ with $\gamma(\calP_1) \subsetneq \gamma(\calP_0)$.
		Since $\gamma(\calP_1)$ misses at least one boundary point of $\gamma(\calP_0)$, which is a point in $P$, we have $|\gamma(\calP_1) \cap P| < |\gamma(\calP_0) \cap P|$.
		This contradicts the choice of $\calP_0$.
	\end{proof}
	
	Let $\calP$ be a colorful partition of $P$ which satisfies \ref{lem3-cond1}, \ref{lem3-cond2} and \ref{lem3-cond3}, and $\gamma(\calP)$ contains a point from a colorful bounding $3$-set in $\calP$.
	In what follows, we choose certain three members in $\calP$ and repartition their union ($9$-point set, three points for each color) into another appropriate $3$-sets, so that if we replace the old ones with the new ones, then we get a desired partition $\calP'$ in Lemma \ref{lem3}. More precisely, let $B$ be a bounding $3$-set in $\calP$ which has a point contained in $\gamma(\calP)$.
	Let $U_1$ and $U_2$ be the unbounding $3$-sets in $\calP$ whose middle points are the boundary points of $\gamma(\calP)$.
	And let $Q = B \cup U_1 \cup U_2$. Under these assumptions, the following claim implies Lemma \ref{lem3}.

	\begin{claim*}\label{lem4}
		There exists a new colorful partition of $Q$ into one bounding and two unbounding $3$-sets, whose set of all middle points is contained in $\gamma(\calP)$ and contains all points in $B \cap \gamma(\calP)$.
	\end{claim*}
	\begin{figure}[ht!]
		\begin{center}
			\begin{subfigure}{0.45\textwidth}
				\centering
				\begin{tikzpicture}[scale=0.9]
				\draw[thick]  (0,0) circle [radius = 2];
				\draw[fill] (0,0) circle [radius = 0.05];
				
				\draw[thick,dotted] (-1.1,1.7)--(-1.8,0.8)--(1.95,-0.5)--(-1.1,1.7);
				\draw[thick,dotted] (1.1,1.7)--(-2,0)--(1.9,0.7)--(1.1,1.7);
				\draw[thick,dotted] (-0.5,1.93)--(0.5,1.93)--(0,-2)--(-0.5,1.93);
				
				\draw[thick] (0,-2)--(-2,0)--(1.1,1.7)--(0,-2);
				
				\draw [red,fill,shift={(-0.5,1.93)}] (0,0.13) -- (0.13,-0.1) -- (-0.13,-0.1) -- (0,0.13);
				\node[above] at (-0.5,1.93) {$b_1$};
				\draw[blue,fill,shift={(0.5,1.93)}] (-0.11,-0.11) rectangle (0.11,0.11);
				\node[above] at (0.5,1.93) {$b_2$};
				\draw[fill,green] (0,-2) circle [radius = 0.13];
				\node[below] at (0,-2) {$b_3$};
				
				\node[right] at (1.1,-1.7) {$-m_1$};
				\draw (1.1,-1.67) circle [radius = 0.05];
				\node[left] at (-1.1,-1.7) {$-m_2$};
				\draw (-1.1,-1.67) circle [radius = 0.05];
				
				\node[right] at (1.1,1.7) {$m_2$};
				\draw [red,fill,shift={(1.1,1.7)}] (0,0.13) -- (0.13,-0.1) -- (-0.13,-0.1) -- (0,0.13);
				\node[left] at (-1.1,1.7) {$m_1$};
				\draw[fill,green] (-1.1,1.7) circle [radius = 0.13];
				
				\node[left] at (-1.8,0.8) {$l_1$};
				\draw [red,fill,shift={(-1.8,0.8)}] (0,0.13) -- (0.13,-0.1) -- (-0.13,-0.1) -- (0,0.13);
				\node[left] at (-2,0) {$l_2$};
				\draw[blue,fill,shift={(-2,0)}] (-0.11,-0.11) rectangle (0.11,0.11);
				
				\node[right] at (1.95,-0.5) {$r_1$};
				\draw[blue,fill,shift={(1.95,-0.5)}] (-0.11,-0.11) rectangle (0.11,0.11);
				\node[right] at (1.9,0.7) {$r_2$};
				\draw[fill,green] (1.9,0.7) circle [radius = 0.13];
				\end{tikzpicture}
				\subcaption{An example of $B'$.}
				\label{fig-claim-1a}
			\end{subfigure}
			~
			\begin{subfigure}{0.45\textwidth}
				\centering
				\begin{tikzpicture}[scale=0.9]
				\draw[thick]  (0,0) circle [radius = 2];
				\draw[fill] (0,0) circle [radius = 0.05];
				
				
				\draw[thin,loosely dashed] (0,-2)--(-2,0)--(1.1,1.7)--(0,-2);
				
				\draw[thick] (-1.8,0.8)--(0.5,1.93)--(1.9,0.7)--(-1.8,0.8);
				\draw[thick] (-1.1,1.62)--(-0.5,1.9)--(1.95,-0.5)--(-1.1,1.7);
				
				\draw [red,fill,shift={(-0.5,1.93)}] (0,0.13) -- (0.13,-0.1) -- (-0.13,-0.1) -- (0,0.13);
				\node[above] at (-0.5,1.93) {$b_1$};
				\draw[blue,fill,shift={(0.5,1.93)}] (-0.11,-0.11) rectangle (0.11,0.11);
				\node[above] at (0.5,1.93) {$b_2$};
				\draw (0,-2) circle [radius = 0.05];
				\node[below] at (0,-2) {$b_3$};
				
				\node[right] at (1.1,-1.7) {$-m_1$};
				\draw (1.1,-1.67) circle [radius = 0.05];
				\node[left] at (-1.1,-1.7) {$-m_2$};
				\draw (-1.1,-1.67) circle [radius = 0.05];
				
				\node[right] at (1.1,1.7) {$m_2$};
				\draw(1.1,1.67) circle [radius=0.05];
				\node[left] at (-1.1,1.7) {$m_1$};
				\draw[fill,green] (-1.1,1.7) circle [radius = 0.13];
				
				\node[left] at (-1.8,0.8) {$l_1$};
				\draw [red,fill,shift={(-1.8,0.8)}] (0,0.13) -- (0.13,-0.1) -- (-0.13,-0.1) -- (0,0.13);
				\node[left] at (-2,0) {$l_2$};
				\draw(-2,0) circle [radius=0.05];
				
				\node[right] at (1.95,-0.5) {$r_1$};
				\draw[blue,fill,shift={(1.95,-0.5)}] (-0.11,-0.11) rectangle (0.11,0.11);
				\node[right] at (1.9,0.7) {$r_2$};
				\draw[fill,green] (1.9,0.7) circle [radius = 0.13];
				\end{tikzpicture}
				\subcaption{Partition via Observation \ref{obs2}.}
				\label{fig-claim-1b}
			\end{subfigure}
		\end{center}
		\caption{A possible case when $|\gamma \cap B|=2$.
			In (a), the thick triangle represents $B'$.
			In (b), two thick triangles represent the partition given by Observation \ref{obs2}. 
			Here, $b_1$ and $b_2$ become the new middle points.}
		\label{fig-claim-1}
	\end{figure}
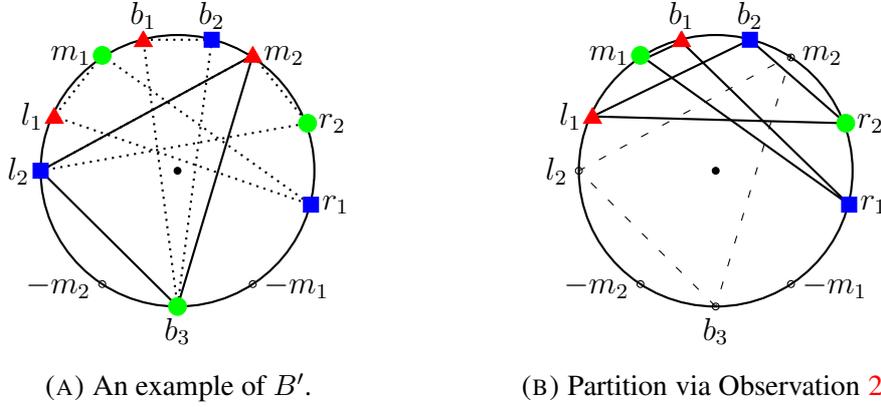
	~
	\begin{figure}[ht!]
		\begin{center}
			\begin{subfigure}[{caption for a}]{0.45\textwidth}
				\centering
				\begin{tikzpicture}[scale=0.9]
				\draw[thick]  (0,0) circle [radius = 2];
				\draw[fill] (0,0) circle [radius = 0.05];
				
				\draw[thick,dotted] (0,2)--(0.55,-2)--(-1.73,-1.1)--(0,2);
				\draw[thick,dotted] (-1.1,1.7)--(-1.8,0.8)--(1.95,-0.5)--(-1.1,1.7);
				\draw[thick,dotted] (1.1,1.7)--(-2,0)--(1.9,0.7)--(1.1,1.7);
				
				\draw[thick] (0.55,-2)--(-1.8,0.8)--(1.95,-0.5)--(0.55,-2);
				
				\draw [red,fill,shift={(0,2)}] (0,0.13) -- (0.13,-0.1) -- (-0.13,-0.1) -- (0,0.13);
				\node[above] at (0,2) {$b_1$};
				\draw[blue,fill,shift={(0.5,-1.93)}] (-0.11,-0.11) rectangle (0.11,0.11);
				\node[right] at (0.55,-2.2) {$b_2$};
				\draw[green,fill] (-1.7,-1.1) circle [radius = 0.13];
				\node[left] at (-1.73,-1.1) {$b_3$};
				\draw (0,-2) circle [radius = 0.05];
				\node[below] at (0,-2) {$-b_1$};
				
				\node[left] at (-1.1,1.7) {$m_1$};
				\draw[blue,fill,shift={(-1.1,1.7)}] (-0.11,-0.11) rectangle (0.11,0.11);
				\node[right] at (1.1,1.7) {$m_2$};
				\draw[blue,fill,shift={(1.1,1.7)}] (-0.11,-0.11) rectangle (0.11,0.11);
				
				\node[right] at (1.1,-1.7) {$-m_1$};
				\draw (1.1,-1.67) circle [radius = 0.05];
				\node[left] at (-1.1,-1.7) {$-m_2$};
				\draw (-1.1,-1.67) circle [radius = 0.05];
				
				\node[left] at (-1.8,0.8) {$l_1$};
				\draw [red,fill,shift={(-1.8,0.8)}] (0,0.13) -- (0.13,-0.1) -- (-0.13,-0.1) -- (0,0.13);
				\node[left] at (-2,0) {$l_2$};
				\draw[green,fill] (-2,0) circle [radius = 0.13];
				
				\node[right] at (1.95,-0.5) {$r_1$};
				\draw[green,fill] (1.95,-0.5) circle [radius = 0.13];
				\node[right] at (1.9,0.7) {$r_2$};
				\draw [red,fill,shift={(1.9,0.7)}] (0,0.13) -- (0.13,-0.1) -- (-0.13,-0.1) -- (0,0.13);
				\end{tikzpicture}
				\subcaption{An example of $B'$ (of type (1)).}
				\label{fig-claim-2a}
			\end{subfigure}
			~
			\begin{subfigure}[{caption for a}]{0.45\textwidth}
				\centering
				\begin{tikzpicture}[scale=0.9]
				\draw[thick]  (0,0) circle [radius = 2];
				\draw[fill] (0,0) circle [radius = 0.05];
				
				
				\draw[loosely dashed] (0.55,-2)--(-1.8,0.8)--(1.95,-0.5)--(0.55,-2);
				
				\draw[thick] (-2,0)--(-1.1,1.7)--(1.9,0.7)--(-2,0);				
				\draw[thick] (-1.73,-1.1)--(0,2)--(1.1,1.7)--(-1.73,-1.1);
				
				\draw [red,fill,shift={(0,2)}] (0,0.13) -- (0.13,-0.1) -- (-0.13,-0.1) -- (0,0.13);
				\node[above] at (0,2) {$b_1$};
				\draw(0.53,-1.95) circle [radius=0.05];
				\node[right] at (0.55,-2.2) {$b_2$};
				\draw[green,fill] (-1.7,-1.1) circle [radius = 0.13];
				\node[left] at (-1.73,-1.1) {$b_3$};
				\draw (0,-2) circle [radius = 0.05];
				\node[below] at (0,-2) {$-b_1$};
				
				\node[left] at (-1.1,1.7) {$m_1$};
				\draw[blue,fill,shift={(-1.1,1.7)}] (-0.11,-0.11) rectangle (0.11,0.11);
				\node[right] at (1.1,1.7) {$m_2$};
				\draw[blue,fill,shift={(1.1,1.7)}] (-0.11,-0.11) rectangle (0.11,0.11);
				
				\node[right] at (1.1,-1.7) {$-m_1$};
				\draw (1.1,-1.67) circle [radius = 0.05];
				\node[left] at (-1.1,-1.7) {$-m_2$};
				\draw (-1.1,-1.67) circle [radius = 0.05];
				
				\node[left] at (-1.8,0.8) {$l_1$};
				\draw(-1.83,0.8) circle [radius=0.05];
				\node[left] at (-2,0) {$l_2$};
				\draw[green,fill] (-2,0) circle [radius = 0.13];
				
				\node[right] at (1.95,-0.5) {$r_1$};
				\draw(1.94,-0.5) circle [radius=0.05];
				\node[right] at (1.9,0.7) {$r_2$};
				\draw [red,fill,shift={(1.9,0.7)}] (0,0.13) -- (0.13,-0.1) -- (-0.13,-0.1) -- (0,0.13);
				\end{tikzpicture}
				\subcaption{Partition via Observation \ref{obs2}.}
				\label{fig-claim-2b}
			\end{subfigure}
		\end{center}
		\caption{A possible case when $|\gamma \cap B| = 1$.
			In (a), the thick triangle represents $B'$.
			In (b), two thick triangles represent the partition given by Observation \ref{obs2}.
			Here, $m_1$ and $b_1$ become the new middle points.
		}
		\label{fig-claim-2}
	\end{figure}
	
	\begin{proof}[Proof of Lemma \ref{lem3}]
		Let $\{B', U'_1, U'_2\}$ be a partition given by the claim, where $B'$ is a bounding, and $U'_1$ and $U'_2$ are unbounding $3$-sets.
		By replacing $\{B, U_1, U_2\}$ with $\{B',U'_1, U'_2\}$, we obtain from $\calP$ another colorful partition $\calP'$ of the whole set $P$.
		Clearly, $\calP'$ satisfies \ref{lem3-cond1} and \ref{lem3-cond2}. Since the two middle points of $\{B',U'_1,U'_2\}$ are contained in $\gamma(\calP)$, we have $\gamma(\calP') \subseteq \gamma(\calP)$.
		Note that at least one of them is a point in $B \cap \gamma(\calP)$.
		So $\gamma(\calP')$ misses at least one boundary point of $\gamma(\calP)$, which is the middle point of either $U_1$ or $U_2$.
		Thus $\gamma(\calP') \subsetneq \gamma(\calP)$.
		
		Now we show $\calP'$ satisfies \ref{lem3-cond3}.
		Choose an arbitrary point $p \in P\cap \gamma(\PP')$.
		If $p$ is in $(P\setminus Q)\cap \gamma(\PP')$, then it is contained in a $3$-set $T$ with $T \in \calP \cap \calP'$.
		Since $\calP$ satisfies \ref{lem3-cond3} and $p\in \gamma(\PP')\subsetneq\gamma(\PP)$, $T$ must be a bounding $3$-set or an unbounding $3$-set whose middle point is $p$.
		
		Now suppose that $p$ is in $Q \cap \gamma(\calP')$.
		Note that $B \cap \gamma(\calP)$ consists of either one or two points in $B$ since $\gamma(\calP)$ cannot contain $B$ by \ref{lem3-cond2} of $\calP$.
		Also, we have $Q \cap \gamma(\calP)=(B\cap\gamma(\PP)) \cup \{m_1,m_2\}$ where $m_1$ and $m_2$ are two boundary points of $\gamma(\calP)$, by \ref{lem3-cond3} of $\calP$.
		For example, see Figure \ref{fig-claim-1a} and \ref{fig-claim-2a} (in the figures, $B$, $U_1$ and $U_2$ are drawn by dotted triangles.)
		
		If there are two points in $B \cap \gamma(\calP)$, then they are exactly the two middle points of $\{B',U'_1,U'_2\}$.
		So neither $m_1$ nor $m_2$ is contained in $\gamma(\calP')$, that is, $Q \cap \gamma(\calP') = B \cap \gamma(\calP)$.
		Thus $p$ must be a middle point of either $U'_1$ or $U'_2$. 
		
		If there is only one point in $B \cap \gamma(\calP)$, then one of the two middle points of $\{B',U'_1,U'_2\}$ is the point in $B \cap \gamma(\calP)$, and the other one is $m_1$ or $m_2$, say $m_1$.
		Since $\gamma(\calP')$ does not contain the other point $m_2$, the points in $Q \cap \gamma(\calP')$ are exactly the two middle points of $\{B', U'_1, U'_2\}$. 
		Therefore, $\calP'$ satisfies \ref{lem3-cond3}.
	\end{proof}
	
	Before we prove the claim, let us outline the proof.
	First, We choose a suitable colorful bounding $3$-set in $Q$ which becomes $B'$ in the proof of Lemma \ref{lem3}.
	After we choose $B'$, we show that the remaining $6$-point set $Q \setminus B'$ can be partitioned into two colorful unbounding $3$-sets, which becomes $U'_1$ and $U'_2$.
	Here we apply Observation \ref{obs2}.
	See Figure \ref{fig-claim-1} and \ref{fig-claim-2}.
	
	Let us note that the condition \ref{lem3-cond1} of $\calP$ directly implies the following.
	\begin{enumerate}[label=($*$)]
		\item\label{minimality} 
		\emph{Any colorful partition of $Q$ contains at most one colorful bounding $3$-set.}
	\end{enumerate}

	Also, as shown in the proof of Lemma \ref{lem3}, the condition \ref{lem3-cond2} of $\calP$ implies that $B \cap \gamma(\calP)$ consists of either one or two points.
	Recall that $\gamma_S(p,q)$ denotes the arc moving from $p$ to $q$ clockwisely.
	
	\medskip

	Now we start to prove the claim. Denote $\gamma = \gamma(\calP)$.
	Let $B = \{b_1, b_2, b_3\}$ with $b_1 \in \gamma$.
	Let $U_1 = \{l_1, m_1, r_1\}$ and $U_2 = \{l_2, m_2, r_2\}$, where $m_i$ denotes the middle point of $U_i$ for $i \in \{1,2\}$. To fix the relative positions of the points in $U_1$ and $U_2$, we assume that $\gamma$ equals $\arc{m_1,m_2}$ and that the shorter arc connecting $l_i$ and $r_i$, which thus contains $m_i$, equals $\arc{l_i,r_i}$ for $i \in \{1,2\}$.
	Since $U_1$ and $U_2$ are unbounding $3$-sets and $\calP$ satisfies \ref{lem3-cond3}, it follows that $l_1, l_2 \in \arc{-m_2,m_1}$ and $r_1, r_2 \in \arc{m_2,-m_1}$.
	The situation is described in Figure \ref{fig-claim-1} and \ref{fig-claim-2}.
	
	Now we consider the two cases according to the number of points in $B \cap \gamma$.
	For simplicity, we assume that $b_i \in C_i$ for $i \in \{1,2,3\}$.
	
	\medskip
	
	First, consider the case when there are two points in $B \cap \gamma$.
	In this case, we assume $b_2 \in \gamma$ and $b_3 \not\in \gamma$.
	Since $b_1$ and $b_2$ are in $\arc{m_1,m_2}$, $b_3$ must be in $\arc{-m_1,-m_2}$.
	See Figure \ref{fig-claim-1}. Suppose that there exists a colorful bounding $3$-set $B'$ in $Q$ which consists of $b_3$ and exactly one point in each of $\{m_1,l_1,l_2\}$ and $\{m_2,r_1,r_2\}$.
	If we denote the clockwise ordering of $Q \setminus B'$ from $b_3$ by $p_1$, $\ldots$, $p_6$, then $\{p_3,p_4\}$ equals $\{b_1, b_2\}$.
	Note that this ordering satisfies the condition in Observation \ref{obs2}. By \ref{minimality}, $Q \setminus B'$ contains no colorful bounding $3$-set.
	So Observation \ref{obs2} gives a partition of $Q \setminus B'$ into two colorful unbounding $3$-sets $U'_1$ and $U'_2$, whose middle points are exactly $b_1$ and $b_2$.
	Therefore, we have a desired partition $\{B',U'_1,U'_2\}$ of $Q$.
	
	Now we show that such $B'$ always exists.
	Note that $b_3$ forms bounding $3$-sets with the pairs $\{m_1, r_1\}$, $\{l_1,r_1\}$, $\{l_2,m_2\}$, $\{l_2,r_2\}$, $\{m_1,m_2\}$ and $\{m_1,r_2\}$.
	If any one of them is colorful, then we are done.
	So suppose otherwise.
	The first four pairs imply that $r_1, l_2 \in C_3$, so that $m_1, m_2, r_2 \not\in C_3$.
	But then either $\{b_3,m_1,m_2\}$ or $\{b_3,m_1,r_2\}$ is colorful, which is a contradiction.
	This concludes the first case.
	
	\medskip
	
	Second, consider the case when there is only one point in $B \cap \gamma$.
	In this case, we assume $b_2 \in \arc{m_2,-b_1}$ and $b_3 \in \arc{-b_1,m_1}$.
	See Figure \ref{fig-claim-2}. As before, we need to find a suitable colorful bounding $3$-set in $Q$.
	Precisely, we need one of either of the following two types.
	\begin{enumerate}[label=(\arabic*)]
		\item\label{type1}
		A colorful bounding $3$-set which consists of exactly one point from $\{l_1,l_2,b_3\}$ and two points in $\{m_2,r_1,r_2,b_2\}$.
		It should contain $b_2$ if $b_2$ lies in $\arc{-m_1,-b_1}$.
		
		\item\label{type2}
		A colorful bounding $3$-set which consists of exactly one point from $\{r_1,r_2,b_2\}$ and two points in $\{m_1,l_1,l_2,b_3\}$.
		It should contain $b_3$ if $b_3$ lies in $\arc{-b_1,-m_2}$.
	\end{enumerate}
	
	Suppose that there exists a colorful bounding $3$-set $B'$ in $Q$ which is of either type \ref{type1} or type \ref{type2}.
	Denote the clockwise ordering of $Q \setminus B'$ from $-b_1$ by $p_1, \ldots, p_6$.
	If $B'$ is of type \ref{type1}, then $\{p_3,p_4\}$ equals $\{m_1,b_1\}$.
	If $B'$ is of type \ref{type2}, then $\{p_3,p_4\}$ equals $\{b_1,m_2\}$.
	In each case, the ordering satisfies the condition of Observation \ref{obs2}. Combined with \ref{minimality}, Observation \ref{obs2} gives a partition of $Q \setminus B'$ into two colorful unbounding $3$-sets $U'_1$ and $U'_2$, whose middle points are exactly $p_3$ and $p_4$.
	Therefore, we have a desired partition $\{B',U'_1,U'_2\}$.
	
	\smallskip
	
	Finally, we show that such $B'$ of either type \ref{type1} or \ref{type2} always exists. By a way of contradiction, suppose that there is no $B'$ of any type. We divide cases according to the position of $b_2$ and $b_3$.
	
	\begin{itemize}
		\item Suppose that neither $b_2$ nor $b_3$ is in $\arc{-m_1,-m_2}$.
		Note that $\{b_2,b_3,m_2\}$ is a bounding $3$-set.
		If it is colorful, then it is of type \ref{type1}, which is forbidden.
		This implies that $m_2$ should not be in $C_1$, so we have either $r_2 \in C_1$ or $l_2 \in C_1$.
		
		If $r_2 \in C_1$, then $\{b_2,b_3,r_2\}$ is colorful.
		If $\arc{b_3,r_2}$ is contained in a semicircle, then $\{b_2,b_3,r_2\}$ is a bounding $3$-set, so that it is of type \ref{type1}.
		Thus $\arc{r_2,b_3}$ must be contained in a semicircle. Since $-m_2$ and $-l_2$ are contained in $\arc{r_2,b_3}$, $\{b_3,m_2,r_2\}$ and $\{b_3,l_2,r_2\}$ are bounding $3$-sets.
		Note that either one of them must be colorful.
		If $\{b_3,m_2,r_2\}$ is colorful, then it is of type \ref{type1}.
		If $\{b_3,l_2,r_2\}$ is colorful, then it is of type \ref{type2}.
		This leads to a contradiction. The case when $l_2 \in C_1$ can be treated in the same way.
		
		\smallskip
		
		\item Suppose that at least one of $b_2$ and $b_3$ is in $\arc{-m_1,-m_2}$.
		Without loss of generality, we assume that $b_2$ is in $\arc{-m_1,-m_2}$.
		More precisely, $b_2$ is in $\arc{-m_1,-b_1}$. Note that $\{b_2,l_2,m_2\}$ and $\{b_2,l_2,r_2\}$ are bounding $3$-sets.
		If any one of them is colorful, then it is of type \ref{type1}, which is forbidden.
		So both of them must not be colorful.
		This implies $l_2 \in C_2$.
		
		Now suppose that $b_3$ is also in $\arc{-m_1,-m_2}$, so it is in $\arc{-b_1,-m_2}$.
		By a symmetric argument to the previous one, we have $r_1 \in C_3$. Note that $\{b_2,l_1,r_1\}$ is a bounding $3$-set.
		If it is colorful, then it is of type \ref{type1}.
		So $l_1$ must be in $C_2$ or $C_3$.
		Since $l_1$ and $r_1$ are in different colors, we have $l_1 \in C_2$.
		Similarly, if $\{b_3,l_2,r_2\}$, being a bounding $3$-set, is colorful, then it is of type \ref{type2}.
		So $r_2$ must be in $C_2$ or $C_3$.
		Since $l_2 \in C_2$, we have $r_2 \in C_3$.			
		Consequently, $m_1$ and $m_2$ are in $C_1$.
		But then we have two disjoint colorful bounding $3$-sets $\{b_2,m_1,r_1\}$ and $\{b_3,l_2,m_2\}$, which contradicts \ref{minimality}.
		
		Now suppose that $b_3$ is not in $\arc{-b_1,-m_2}$.
		Then $b_3$ is in $\arc{-m_2, m_1}$.
		Note that $\{b_2,b_3,m_2\}$ is a bounding $3$-set.
		If it is colorful, then it is of type \ref{type1}.
		So $m_2$ must be in $C_2$ or $C_3$.
		Since $l_2 \in C_2$, we have $m_2 \in C_3$, so $r_2 \in C_1$.
		With this coloring, both $\{b_2,b_3,r_2\}$ and $\{l_2,r_2,b_3\}$ are colorful. Note that those two 3-sets have common points $r_2$ and $b_3$. If $-r_2$ is in $\arc{-m_2,b_3}$, then $\{b_2,b_3,r_2\}$ is a bounding $3$-set, so it is of type \ref{type1}.
		So $-r_2$ must be in $\arc{b_3, m_1}$.
		But then $\{b_3,l_2,r_2\}$ is a bounding $3$-set, so it is of type \ref{type2}.
		This contradiction concludes the proof.
	\end{itemize}
	
	\section{Final remarks}\label{section-final}
	In this section, we discuss whether it is possible to generalize Theorem \ref{thm-line} in two different aspects. In Section \ref{section-3d}, we show that there is a set of colored planes in convex position in $\R^3$ such that in every colorful partition of the set, simplices determined by parts of the partition have an empty intersection. Hence a direct generalization  of Theorem \ref{thm-line} for hyperplanes in convex position in $\R^d$ seems impossible. And in Section \ref{section-nece}, we see an example of lines in the plane showing that being in convex position is not a necessary condition for the conclusion of Theorem \ref{thm-line}.
	
	\subsection{An example in 3-dimension}\label{section-3d}
	Similar with lines in the plane, we say that hyperplanes $h_1, \dots, h_n$ in general position in $\R^d$ are \textit{in convex position} if there is a connected component of the complement $\R^d \setminus (\bigcup_{i=1}^n h_i)$ whose boundary meets every hyperplane. Especially, planes in general position in $\R^3$ and tangent to an open hemisphere are in convex position, since the connected component containing the origin has a boundary which meets every plane at its tangent point.
	
	\medskip
	We choose the following 8 points
	\begin{table}[ht]
		\centering
		\begin{tabular}{lll}
			$p_1 = (1/3, 2/3, 2/3)$ & $p_2 = (7/9, -4/9, 4/9)$ & $p_3 = (6/7, 2/7, 3/7)$\\
			$p_4 = (17/19, 6/19, 6/19)$&
			$p_5 = (1/3, -2/3, 2/3)$&
			$p_6 = (2/3, -2/3, 1/3)$\\
			$p_7 = (6/7, 3/7, 2/7)$&
			$p_8 = (-2/7, 3/7, 6/7)$
		\end{tabular}
	\end{table}\\
	from the open hemisphere $S=\{(x,y,z)\in\R^3\,:\, x^2+y^2+z^2=1\textrm{ and } z>0\}$. For $1\leq i \leq 8$, let $h_i$ be a plane tangent to $S$ at $p_i$, and define $\mathcal{H}=\{h_i\,:\,1\leq i\leq 8\}$. It is straightforward but tedious to check that $\mathcal{H}$ is in general position. Make a partition of $\mathcal{H}$ into color classes
	\begin{align*}
	C_1 = \{h_1, h_5\},\,
	C_2 = \{h_2, h_6\},\,
	C_3 = \{h_3, h_7\}\, \textrm{ and }C_4 = \{h_4, h_8\}.
	\end{align*}
	
	We claim that for every colorful partition of $\mathcal{H}$, the simplices determined by parts of the partition have an empty intersection. Let $h_i^+$ (or $h_i^-$) be the closed halfspace bounded by $h_i$ and not containing (or, containing, respectively) the origin. For every quadruple $(i_1, i_2,i_3,i_4)$ of distinct elements from $\{1, \dots, 8\}$, and every sign vector $w$ in $\{+,-\}^8$, let $w|_{(i_1, i_2, i_3, i_4)}$ be the closed region $\bigcap_{j=1}^4 h_{i_j}^{w(i_j)}$ where $w(i)$ is the $i$th coordinate of $w$. Let $v(i_1, i_2, i_3, i_4)\in \{+,-\}^8$ be the unique sign vector such that $v(i_1, i_2, i_3, i_4)|_{(i_1, i_2, i_3, i_4)}$ is the bounded simplex determined by planes $h_{i_1}$, $h_{i_2}$, $h_{i_3}$ and $h_{i_4}$, and $v(i)=0$ for every $i \notin \{i_1, \dots, i_4\}$. Note that $-v(i_1, i_2, i_3, i_4)|_{(i_1, i_2, i_3, i_4)}$ is an emptyset.
	
	Let us consider a partition $\{h_{i_1},h_{i_2},h_{i_3},h_{i_4}\}\cup\{h_{i_5},h_{i_6},h_{i_7},h_{i_8}\}$ of $\HH$ . If there is a quadruple $(j_1, j_2,j_3,j_4)$ of distinct elements such that the vector $v(i_1,i_2,i_3,i_4)$ $+$ $v(i_5,i_6,i_7,i_8)$ in positions $j_1,j_2,j_3,j_4$ has the same entries as 
	$-v(j_1,j_2,j_3,j_4)$ for a certain choice of positions $j_1,j_2,j_3,j_4$, then we can see that the intersection of two simplices determined by the subsets must be empty, since the intersection is contained in $-v(j_1, j_2, j_3, j_4)|_{(j_1, j_2, j_3, j_4)}$ which is empty. Therefore, it is sufficient to find such a quadruple for every colorful partition. Which quadruple $(j_1,j_2, j_3, j_4)$ can be chosen for each colorful partition is summarized at Table \ref{table2}, with corresponding sign vectors.
	
	\begin{table}[ht]
		\centering
		\begin{tabular}{|c|c|c|}
			\hline
			$(i_1,i_2,i_3,i_4,i_5,i_6,i_7,i_8)$ & $v(i_1,i_2,i_3,i_4)+v(i_5,i_6,i_7,i_8)$ &
			$(j_1,j_2,j_3,j_4)$ \\
			\hline
			$(1,2,3,4,5,6,7,8)$ &$(-,-,+,-,+,-,+,-)$ &$(1,2,5,7)$\\
			$(1,2,3,8,5,6,7,4)$ &$(+,+,-,+,-,-,-,-)$ &$(1,2,3,4)$\\
			$(1,2,7,4,5,6,3,8)$ &$(+,-,+,+,+,-,-,-)$ &$(2,4,5,8)$\\
			$(1,2,7,8,5,6,3,4)$ &$(+,+,+,-,-,+,-,-)$ &$(1,4,5,6)$\\
			$(1,6,3,4,5,2,7,8)$ &$(-,+,+,-,-,-,-,+)$ &$(1,2,5,8)$\\
			$(1,6,3,8,5,2,7,4)$ &$(-,-,+,+,+,-,-,+)$ &$(1,2,3,8)$\\
			$(1,6,7,4,5,2,3,8)$ &$(-,-,+,+,+,-,-,-)$ &$(1,2,4,5)$\\
			$(1,6,7,8,5,2,3,4)$ &$(+,+,+,-,-,+,-,-)$ &$(1,4,5,6)$\\
			\hline
		\end{tabular}
		\caption{For all colorful partitions, the sign vector $-v(j_1,j_2,j_3,j_4)$ guarantees that simplices formed by the parts have empty intersection.}
		\label{table2}
	\end{table}

	\subsection{Being in convex position is not necessary.}
	\label{section-nece}
	\begin{figure}[ht]
		\begin{center}
			\begin{tikzpicture}[scale=0.8]
			\draw[thick] (-5,2.5) -- (5,2.5);
			\draw[thick] (1,5) -- (-4.5,-3);
			\draw[thick] (0,5) -- (4,-3);
			\draw[thick] (3.5,5) -- (1,-3);
			\draw[thick] (5,-1) -- (-5,-2);
			\draw[thick] (5,-3) -- (-4,3);
			
			\node[below] at (-1,-1) {$l_6$};
			\node[below] at (-0.5,0.7) {$l_1$};
			\node[left] at (-1.05,2) {$l_2$};
			\node[above] at (0.4,2.5) {$l_3$};
			\node[right] at (1.65,1.8) {$l_4$};
			\node[right] at (1.8,-0.3) {$l_5$};
			
			\draw[fill=gray!50] (-0.71,2.5)--(1.25,2.5)--(2.15,0.68)--(1.68,-0.78)--(-1.5,1.33);
			\end{tikzpicture}
		\end{center}
		\caption{An example not in convex position, which has a colorful partition with intersecting triangles for every possible coloring.}
		\label{counterplane2}
	\end{figure}
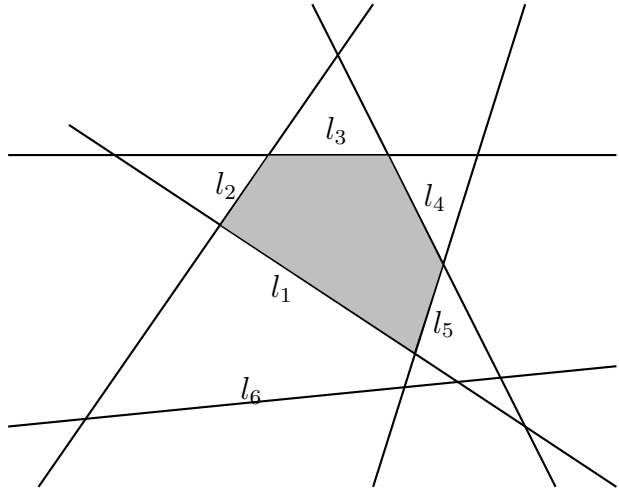

	In this subsection, we show that being in convex position is not necessary for the conclusion of Theorem \ref{thm-line}. More precisely, we give an example of 6 lines which is not in convex position such that for every coloring using 3 color classes of equal size, we can always find a partition into colorful parts which determine intersecting triangles. This suggests finding a broader class of line arrangments where the conclusion of Theorem \ref{thm-line} holds.

	See Figure \ref{counterplane2}. There are 6 lines $l_1, \dots, l_6$ not in convex position, and the 5 lines $l_1, \dots, l_5$ bounds the shaded region. We color each line with colors RED, BLUE and GREEN so that each color class has size 2. Without loss of generality, we assume that $l_6$ is painted by BLUE.

	First, we suppose that $l_2$ and $l_6$ are in different colors. We assume that $l_2$ is colored by RED. If $l_3$ is colored by GREEN, then we can choose a partition $\{l_2, l_3,l_6\}\cup \{l_1,l_4,l_5\}$. If $l_4$ is colored by GREEN, then we can choose a partition $\{l_2,l_4,l_6\}\cup\{l_1,l_3,l_5\}$. So, we only need to consider the case when both $l_1$ and $l_5$ are colored by GREEN. But then, we can take a partition $\{l_2, l_5,l_6\}\cup \{l_1,l_3,l_4\}$.
	
	Next, suppose $l_2$ and $l_6$ are in the same color BLUE.
	If $l_1$ and $l_5$ have different colors, then we can take a partition $\{l_1,l_2,l_5\}\cup\{l_3,l_4,l_6\}$. In this partition, we have a common intersection between triangles formed by 3-sets $\{l_1,l_2,l_5\}$ and $\{l_3,l_4,l_6\}$ in the region bounded by lines $l_3$, $l_4$ and $l_5$. So we can assume that $l_1$ and $l_5$ have the same color. Then, we take a partition into $\{l_1,l_3,l_6\}\cup\{l_2,l_4,l_5\}$. In this case again, we have a common intersection in the region bounded by lines $l_3$, $l_4$ and $l_5$.
	
	\section*{Acknowledgements}
	We are grateful to Andreas Holmsen, Roman Karasev,
	Liping Yuan, Pablo Sober\'{o}n, Minki Kim and Hong Chang Ji for their excellent comments.


\begin{thebibliography}{100}
		\bibitem{history1} I.~B\'{a}r\'{a}ny, Z.~F\"{u}redi and L.~Lov\'{a}sz: On the number of halving planes, {\em Combinatorica} {\bf 10(2)} (1990), 175–183.
		
		\bibitem{history2} I.~B\'{a}r\'{a}ny and D.G.~Larman: A colored version of Tverberg's theorem, {\em J. London Math. Soc.} {\bf s2-45(2)} (1992), 314-320.
		
		\bibitem{join1} I.~B\'{a}r\'{a}ny, A.~Holmsen and R.~Karasev: Topology of geometric joins, {\em Discrete Comput. Geom.} {\bf 53} (2015), 402-413.
		
		\bibitem{history4} P.V.M.~Blagojevi\'{c}, B.~Matschke and G.M.~Ziegler: 
		Optimal bounds for the colored Tverberg problem, {\em J. Eur. Math. Soc.} {\bf 17} (2015), 739–754.
		
		\bibitem{join2} J.-D.~Boissonnat, O.~Devillers and F.P.~Preparata: Computing the union of $3$-colored triangles, {\em Intern. J. Comput. Geom. Appl.} {\bf 1} (1991), 187–196.
		
		\bibitem{hall} P.~Hall: On Representatives of Subsets, {\em J. London Math. Soc.} {\bf s1-10(1)} (1935), 26–30.
		
		\bibitem{dualcentral} R.N.~Karasev: Dual central point theorems and their generalizations, {\em Sbornik: Mathematics} {\bf 199(10)} (2008), 1459–1479.
		
		
		\bibitem{roudneff} J.-P.~Roudneff: Tverberg-type theorems for pseudoconfigurations of points in the plane, {\em European J. Combin.} {\bf 9(2)} (1988), 189-198.
		
		
		\bibitem{join3} A.~Schulz and C.~T\'{o}th: The union of colorful simplices spanned by a colored point set, {\em Comput. Geom. Theory Appl.} {\bf 46} (2011), 574–590.
		
		
		\bibitem{history0} H.~Tverberg: A generalization of Radon's theorem, {\em J. London Math. Soc.}, {\bf s1-41(1)} (1966), 123–128.
		
		\bibitem{yuan} Liping Yuan. Personal communication.
		
		\bibitem{history3} R.T.~\v{Z}ivaljevi\'{c} and S.T.~Vre\'{c}ica: The colored Tverberg's problem and complexes of injective functions, {\em J. Comb. Theory, Ser. A} {\bf 61(2)} (1992), 309-318.
		
		\bibitem{user} R.T.~\v{Z}ivaljevi\'{c}: User's guide to equivariant methods in combinatorics, {\em II. Publ. Inst. Math.} {\bf 64(78)} (1998), 107-132.
		
	\end{thebibliography}
\end{document}